\documentclass[a4paper,11pt]{article}
\usepackage[utf8]{inputenc}
\usepackage{mathrsfs}
\usepackage{amsmath}
\usepackage{amsfonts}
\usepackage{amsthm}
\usepackage{graphicx}

\newtheorem{thm}{Theorem}[section] 
\newtheorem{question}{Question}
\theoremstyle{definition}

\newtheorem{remark}[thm]{Remark}%[section]
\newtheorem{conjecture}[thm]{Conjecture}

\title{On percolation in Poisson graphs}
\author{Johan Bj\"orklund\footnote{Department of Mathematics, Uppsala University, SE-75310 Uppsala, Sweden and Department of Mathematics, Université Pierre et Marie Curie, 75005 Paris, France. Supported  by the Swedish Research Council.}
	\and Victor Falgas-Ravry\footnote{Department of Mathematics, Vanderbilt University, 
1326 Stevenson Center, Nashville, TN 37240, USA.}\and
Cecilia Holmgren\footnote{Department of Mathematics, Stockholm University, 114 18 Stockholm, Sweden. Supported  by the Swedish Research Council.}}
\begin{document}

\maketitle

\begin{abstract}
Equip each point $x$ of a homogeneous Poisson process $\mathcal{P}$  on $\mathbb{R}$ with $D_x$ edge stubs, where the $D_x$ 
are i.i.d. positive integer-valued random variables with distribution given by $\mu$. 
Following the \emph{stable multi-matching scheme} introduced by Deijfen, H\"aggstrom and Holroyd~\cite{Deijfen}, 
we pair off edge stubs in a series of rounds to form the edge set of an infinite component $G$ on the vertex set $\mathcal{P}$. 
In this note, we answer questions of Deijfen, Holroyd  and Peres~\cite{Deijfen2} and Deijfen, H\"aggstr\"om and Holroyd~\cite{Deijfen} 
on percolation (the existence of an infinite connected component) in $G$. We prove that percolation may occur a.s.\ even if $\mu$ has support over odd integers. 
Furthermore, we show that for any $\varepsilon>0$ there exists a distribution $\mu$ such that $\mu(\{1\})>1-\varepsilon$ such that percolation still occurs a.s..
%Previous version of the abstract: In this paper we answer questions of Deijfen, Holroyd  and Peres~\cite{Deijfen2} and Deijfen, H\"aggstr\"om and Holroyd~\cite{Deijfen} on the existence of infinite components in \emph{Poisson graphs} in one dimension, depending on the degree distributions of half-edges.  Let each point of a homogeneous Poisson process on $\mathbb{R}$ independently be equipped with a random number $D$ of stubs (half-edges) according to a given probability distribution $\mu$ on the positive integers. In \cite{Deijfen} the \emph{stable multi-matching scheme} was introduced, describing a natural matching scheme for perfectly matching thebstubs to obtain a simple graph with degree distribution $\mu$.  We prove that it is possible to havebinfinite components even if $D$ only takes odd integers as values; \cite{Deijfen2} suggests that there is no infinite component if $\mu$ is distributed on some odd integers. In \cite{Deijfen} there was a similar question: Does there exist some $k<1$ such that if $P(D=1)>k$ then no infinite component will exist? We show that for any $k<1$ there are degree distributions with $P(D=1)>k$ resulting in infinite components. 
\end{abstract}

\section{Introduction}In this paper, we study certain matching processes on the real line. Let $D$ be a random variable with distribution $\mu$ supported on the positive integers. Generate a set of vertices $\mathcal{P}$ by a Poisson point process of intensity $1$ on $\mathbb{R}$. Equip each vertex $x\in \mathcal{P}$ with a random number $D_x$ of edge stubs, where the $(D_x)_{x\in\mathcal{P}}$ are i.i.d. random variables with distribution given by $D$. Now form edges in rounds by matching edge stubs in the following manner. In each round, say that two vertices $x,y$ are \emph{compatible} if they are not already joined by an edge and both $x$ and $y$ still possess some unmatched edge stubs. Two such vertices form a \emph{mutually closest compatible pair} if $x$ is the nearest $y$-compatible vertex to $y$ in the usual Euclidean distance and vice-versa. For each such mutually closest compatible pair $(x,y)$, remove an edge stub from each of $x$ and $y$ to form the edge $xy$. Repeat the procedure indefinitely.

This matching scheme, known as \emph{stable multi-matching}, was introduced by Deijfen, H\"aggstr\"om and Holroyd \cite{Deijfen}, who showed that it a.s.\ exhausts the set of edge stubs, yielding an infinite graph $G=G(\mu)$ with degree distribution given by $\mu$.

A natural question to ask is which degree distributions $\mu$ (if any) yield an infinite connected component in $G$. For example if $\mu(\{1\})=1$, 
then no such component exists, while if $\mu(\{2\})=1$, Deijfen, Holroyd and Peres~\cite{Deijfen2} 
suggest  that percolation (the existence of an infinite component) occurs a.s.. Note that by (a version of) Kolmogorov's zero--one law, the probability of percolation occurring is zero or one.  Also, as shown by Deijfen, Holroyd and Peres~(see~\cite{Deijfen2}, Proposition 1.1), an infinite component in $G$, if it exists, is almost surely unique.

Taking the Poisson point process in $\mathbb{R}^d$ for some $d\geq 1$ and applying the stable multi-matching scheme mutatis mutandis, we obtain the $d$-dimensional Poisson graph $G_d$. Deijfen, H\"aggstr\"om and Holroyd proved the following result on percolation in $G_d$:
\begin{thm} \label{mainDeijfen} (Deijfen, H\"aggstr\"om and Holroyd \cite[Theorem 1.2]{Deijfen})
\begin{enumerate}
\item[(i)] For all $d\geq 2 $ there exists $k=k(d) $ such that if $\mu(\{n\in \mathbb{N}:\ n\geq k\})=1$, then a.s.\ $G_d$ percolates.
\item[(ii))] For all $ d\geq 1 $, if $\mu(\{1,2\})=1$ and $\mu(\{1\})>0 $, then a.s.\ $G_d$ does not percolate.
\end{enumerate}
\end{thm}
Their proof of part (i) of Theorem~\ref{mainDeijfen} relies on a comparison of the $d$-dimensional stable multi-matching process with dependent site percolation on $\mathbb{Z}^d$. In particular, since the threshold for percolation in $\mathbb{Z}$ is trivial, their argument cannot say anything about percolation in the $1$-dimensional Poisson graph $G=G_1$.

Related to part (ii) of Theorem~\ref{mainDeijfen}, Deijfen, H\"aggstr\"om and Holroyd asked the following question.
\begin{question}[Deijfen, H\"aggstr\"om and Holroyd]\label{question2}
Does there exist some $\varepsilon>0$ such that if $\mu(\{1\})>1-\varepsilon$, then a.s.\ $G_d$ contains no infinite component?
 \end{question}
In subsequent work on $G=G_1$, Deijfen, Holroyd and Peres~\cite{Deijfen2} observed that simulations suggested percolation might not occur when $\mu(\{3\})=1$, and asked whether the presence of odd degrees kills off infinite components in general. 
\begin{question}[Deijfen, Holroyd and Peres]\label{question1}
 Is it true that percolation in $G=G_1$ occurs a.s., if and only if, $\mu$ has support only on the even integers?
 \end{question}
In this paper we prove the following theorem, answering Question~\ref{question1} negatively:
\begin{thm}\label{main}
There exist degree distributions $\mu$ with support on the odd integers, such that the stable multi-matching process a.s.\ yields an infinite component.
\end{thm}
Furthermore, we show in Remark \ref{answerquestion2} that for any $\varepsilon>0$ we can construct such a degree distribution $\mu$ with  $\mu(\{1\})>1-\varepsilon$, thus also answering Question~\ref{question2} negatively.  We note however that the distribution $\mu$ we construct has unbounded support; it would be interesting to find an example with bounded support only.

\section{Proof of Theorem \ref{main}}
The idea behind the proof of Theorem~\ref{main} is to set $\mu(\{d_i\})=1/2^i$ for a sharply increasing sequence of integers $(d_i)_{i \in \mathbb{N}}$. Suppose we are given a vertex $x_i$ with degree $D_{x_i}=d_i$. By choosing $d_i$ large enough we can ensure that with probability close to $1$, there exists some vertex $x_{i+1}$ with $D_{x_{i+1}}=d_{i+1}$ that is connected to $x_i $ by an edge of $G$. 
Let $E_i $, $ i\geq 1 $, be the event that a given vertex $x_i$ of degree $d_i$ is connected to some vertex $x_{i+1}$ of degree $d_{i+1} $. Starting from a vertex $x_1$ of degree $d_1$, we see that if $\bigcap_{i=1}^{\infty} E_i$ occurs, then there is an infinite path $x_1x_2x_3\ldots$  in $G$. If the events $(E_i)_{i \in \mathbb{N}}$ were independent of each other, then $\mathbb{P}(\bigcap_{i=1}^{\infty} E_i) = \prod_{i \in \mathbb{N}} \mathbb{P}(E_i)$, which we could make strictly positive by letting the sequence $(d_i)_{i \in \mathbb{N}}$ grow sufficiently quickly, ensuring in turn that percolation occurs a.s.. Of course the events $(E_i)_{i \in \mathbb{N}}$ as we have loosely defined them above are highly dependent. We circumvent this problem by working with a sequence of slightly more restricted events, for which we do have full independence. As we have no restrictions on the $d_i$ other than their growth rate, each of them can be chosen to be odd or even as we please.

Before we begin the proof, let us introduce the following notation. Given $x\in\mathcal{P}$, let $B(x,r)$ be the collection of all vertices in $\mathcal{P}$ within distance at most $r$ of $x$. We say that a pair of vertices $(x,y)$ with degrees $(D_x, D_y)$ is \emph{strongly connected} if $\vert B(x,\vert y-x\vert )\vert \leq D_x$ and $\vert B(y,\vert y-x \vert)\vert\leq D_y$. Observe that if a pair of vertices $(x,y)$ is strongly connected, then there will a.s.\ be an edge between $x$ and $y$ in the stable multi-matching scheme. %it a.s. will form an edge under the stable multi-matching scheme.
\begin{figure}\begin{center}
		\includegraphics[scale=0.4]{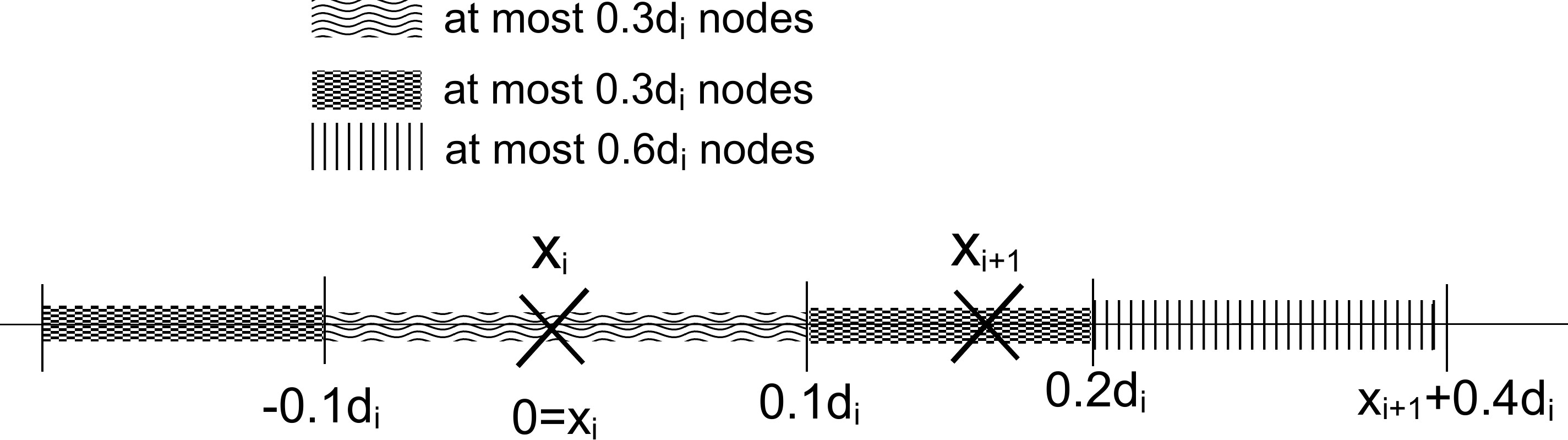}
		\caption{Restrictions on the number of nodes in various intervals when the event $D_i(x_i)$ occurs.}
	\end{center}
\end{figure}
\begin{proof}[Proof of Theorem~\ref{main}]
Set $d_i=10\cdot4^i$ and $\mu(\{d_i\})=\frac{1}{2^i}$ for each $i\in \mathbb{N}$.
Suppose that we condition on a particular vertex $x_i$ of degree $d_i$ belonging to the point process $\mathcal{P}$. (Formally, we consider a Palm version of $\mathcal{P}$, i.e. a version of $\mathcal{P}$ conditioned on $x_i \in \mathcal{P}$, with the rest of the process taken as a stationary background. Note that the Palm version of a homogeneous Poisson process has the same distribution as the unconditioned process itself with an added point; see~\cite[Chapter 11]{Kallenberg}.) 
%for a rigorous exposition of the theory of Palm processes.) 
We condition further on there being at most $0.3 d_i$ points of $\mathcal{P}$ in the interval of length $0.2d_i$ centered at $x_i$, and write $F_i(x_i)$ for the event we are conditioning on.  
%within distance $0.1 d_i$ of $x_i$. 
By the standard properties of Poisson point processes, conditioning on $F_i(x_i)$ does not affect the probability of any event defined outside the interval. %this interval around $x_i$.

Let $A_i(x_i)$ be the event that there is a vertex $x_{i+1} \in \mathcal{P}$ with degree $d_{i+1}$ such that $0.1d_i< \vert x_{i+1} -x_i\vert < 0.2 d_i$. Viewing $\mathcal{P}$ as the union of two thinned Poisson point processes, one of intensity $2^{-(i+1)}$ giving us the vertices of degree $d_{i+1}$ and another of intensity $1-2^{-(i+1)}$ giving us the rest of the vertices,  we see that $\mathbb{P}({\left(A_i(x_i)\right)}^c)=e^{-\frac{0.2d_i}{2^{i+1}}} =e^{-2^i}$. If $A_i(x_i)$ occurs, let $x_{i+1}$ denote the a.s.\ unique vertex of degree $d_{i+1}$ which is nearest to $x_i$ among those degree $d_{i+1}$ vertices lying at distance at least $0.1 d_i$ from $x_i$.

Let $B_i(x_i)$ be the event that there are at most $0.3 d_i$ vertices $x \in \mathcal{P}$ with $0.1 d_i< \vert x-x_{i}\vert<0.2 d_i$. Furthermore, given a point $y$  (not necessarily belonging to our process $\mathcal{P}$) with $0.1d_i< \vert y -x_i\vert < 0.2 d_i$, we let $C_i(x_i, y)$ be the event that there are at most $0.6d_i$ vertices $x$ lying at distance at least $0.2 d_i$ from $x_i$ and at most $0.4 d_i$ from $y$. A quick calculation (using the Chernoff bound) yields that $\mathbb{P}({B_i(x_i)}^c)=e^{-3(\log \frac{3}{2}-\frac{1}{3})4^i+O(i)}$ and $\mathbb{P}({C_i(x_i, y)}^c)\leq e^{-6(\log \frac{3}{2}-\frac{1}{3})4^i+O(i)}$ respectively.

Finally let $D_i(x_i)= A_i(x_i)\cap B_i(x_i) \cap C_i(x_i, x_{i+1})$.  If $D_i(x_i)$ occurs, then $x_i$ and $x_{i+1}$ are strongly connected, since our initial assumption $F_i(x_i)$ together with $B_i(x_i)$ tells us that $\vert B(x_i, \vert x_i -x_{i+1}\vert)\vert \leq 0.6 d_i$, while $F_i(x_i)$ together with $B_i(x_i)\cap C_i(x_i, x_{i+1})$ yield that $\vert B(x_{i+1}, 0.4 d_{i})\vert \leq 1.2 d_i =0.3 d_{i+1}$ (see Figure~1). This last statement is exactly our initial conditioning $F_i(x_i)$ with $i$ replaced by $i+1$; hence $D_i(x_i)\cap F_i(x_i)\subseteq F_{i+1}(x_{i+1}) $. %and $x_i$ by $x_{i+1}$.

By the union bound, we have 
\begin{align*}
\mathbb{P}\big(D_i(x_i)\vert F_i(x_i)\big)&\geq  1-\mathbb{P}\big({\left(A_i(x_i)\right)^c\vert F_i(x_i)}\big) -\mathbb{P}\big({\left(B_i(x_i)\right)^c\vert F_i(x_i)}\big)\\
&\qquad -\sup_{y:\ \vert y -x_i\vert \in (0.1d_i, 0.2d_i)}\mathbb{P}\big({\left(C_i(x_i, y)\right)^c\vert F_i(x_i)}\big)\\
&>1- e^{-2^i}(1+o(1)).
\end{align*}
Selecting $i_0$ sufficiently large and some arbitrary vertex $x_{i_0}$ of degree $d_{i_0}$ as a starting point, we may define events $D_{i_0}(x_{i_0}), D_{i_0+1}(x_{i_0+1}), \ldots$ inductively, each conditional on its predecessors, with
\begin{align*}\mathbb{P}\big(\bigcap_{i \geq i_0} D_i(x_i)\vert F_{i_0}(x_{i_0})\big)&=\prod_{i\geq i_0}\mathbb{P}\big(D_i(x_i)\vert \cap_{j<i} D_j(x_j)\cap F_{i_0}(x_{i_0})\big) \\
&=\prod_{i\geq i_0}\mathbb{P}\big(D_i(x_i)\vert F_i(x_i)\big)
>1 - 2\sum_{i\geq i_0}e^{-2^i}>0. 
\end{align*}
From any vertex $x_{i_0}\in \mathcal{P}$ there is, with strictly positive probability, an infinite path in $G$, $x_{i_0}, x_{i_0+1}, \ldots$ through vertices of increasing degrees $d_{i_0}, d_{i_0+1}, \ldots $ . It follows that $G$ a.s.\ contains a strongly connected infinite component.
\begin{remark}
	The pairs $(x_{i_0},x_{i_0+1}), (x_{i_0+1},x_{i_0+2}), \ldots$ remain strongly connected if we increase the degrees. Thus, if a given degree distribution $\mu$ a.s.\ results in a strongly connected infinite component in $G(\mu)$, then any degree distribution $\mu'$ that stochastically dominates $\mu$ will also a.s.\ yield a strongly connected infinite component in $G_(\mu')$.
	\end{remark} 
Thus, if we set $d_i'=d_i +1$ and $\mu'(d_i')=2^{-i}$, then the associated Poisson graph $G=G(\mu')$ a.s.\ percolates, though all vertices have odd degrees.
\end{proof}
\section{Concluding remarks}
\begin{remark}\label{answerquestion2}
Note that our proof of Theorem~\ref{main} does not use any information about $d_i$ for $i<i_0$. In particular, we could set $\mu(\{1\})=1 - 2^{-i_0+1}$ and $\mu(\{d_i\})=2^{-i}$ for $i \geq i_0$ and still have a distribution for which $G$ percolates a.s.. Choosing $i_0$ sufficiently large, this implies a negative answer to Question~\ref{question2}.
\end{remark}
\begin{remark}\label{dim2}
%. There again for every $\varepsilon>0$ we could let $\mu(\{1\})>1-\varepsilon$ and still ensure a.s. percolation of $G_d$. 
The existence of degree distributions that a.s.\ result in an infinite component in dimensions $d\geq2$ was established in~\cite[Theorem 1.2 a)]{Deijfen}. Our proof of Theorem~\ref{main} for $G=G_1(\mu)$ easily adapts to higher dimensions $d\geq 2$ (with $d$-dimensional balls and annuli replacing intervals and punctured intervals, and the sequence $(d_i)_{i\in\mathbb{N}}$ being scaled accordingly), giving a different approach to the construction of examples in that setting.
\end{remark}
The distribution $\mu$ we construct in Theorem~\ref{main} has unbounded support, and the expected degree of a vertex in $G(\mu)$ is infinite. We believe however that the answer to Questions~\ref{question2} and~\ref{question1} should still remain negative if $\mu$ is required to have bounded support. %Our motivation is the following.
Indeed we conjecture the following:
\begin{conjecture}\label{bold conjecture}
	For every $\varepsilon>0$, there exists $k=k(\varepsilon)$ such that if $\mu(\{n\in\mathbb{N}: \ n\geq k\})>\varepsilon$, then percolation occurs a.s.\ in $G=G_1(\mu)$.
\end{conjecture}
One might expect that there is a \emph{critical value} $d_{\star}$ of the expected degree for percolation.
We believe however that no such critical value exists: 
\begin{conjecture}	
	There is no critical value $d_{\star}$ such that if $\mathbb{E}(D)<d_{\star}$ then a.s.\ percolation does not occur, while if $\mathbb{E}(D)>d_{\star}$, then a.s.\ percolation occurs in the stable multi-matching scheme.
\end{conjecture}
Let us give some motivation for this conjecture. By~\cite[Theorem 1.2 b)]{Deijfen}, for any $\mu$ with support on $\{1,2\}$ and $\mu(\{1\})>0$, $G_1(\mu)$ a.s.\ does not percolate. So any putative critical value must satisfy $d_{\star}\geq 2$. Now, pick $\varepsilon>0$ and choose $\delta \gg d_{\star}$. Let $\mu$ be a degree distribution with support on $\{1, \delta\}$, such that the expected degree satisfies $\mathbb{E} (D)<d_{\star}-\varepsilon$. By definition of $d_{\star}$ this would imply $G(\mu)$ a.s.\ does not percolate. Assign degrees independently at random to the vertices of $G(\mu)$. Perform the first $\delta/2$ stages of the stable multi-matching process. By then most degree $1$ vertices have been matched (and in fact matched to other degree $1$ vertices). Now force the remaining degree $1$ vertices to match to their future partners. Consider the vertices that had originally been assigned $\delta$ edge stubs. A number of these edge stubs will have been used up by the process so far, and the number of edge stubs left at each vertex is not independent; nevertheless we expect most degree $\delta$ vertices will have at least $\delta/4$ edge stubs left, and that the number of stubs left will be almost independently distributed. Thus, we believe the stable multi-matching scheme on the remaining edge stubs of the degree $\delta$ vertices will contain as a subgraph the edges of a stable multi-matching scheme on a thinned Poisson point process on $\mathbb{R}$ corresponding to the degree $\delta$ vertices, and with degrees given by some random variable $D'$ with $\mathbb{E}(D')>\delta/4\gg d_{\star}$. Since rescaling does not affect the stable multi-matching process, this would imply $G(\mu)$ a.s.\ percolates (by definition of $d_{\star}$), a contradiction.
\subsection*{Acknowledgements} 
We would like to thank Svante Janson for valuable comments.
 

\begin{thebibliography}{99}
 \bibitem{Deijfen}Deijfen M., H\" aggstr\"om O. and Holroyd A.E., Percolation in invariant Poisson graphs with i.i.d. degrees, \emph{Ark. Mat.}, \textbf{50} (2012), 41--58.
 \bibitem{Deijfen2}Deijfen M., Holroyd A.E. and Peres Y., Stable Poisson Graphs in One Dimension, \emph{Electronic Journal of Probability} \textbf{16} (2011) 1238--1253.
  \bibitem{Holroyd} Holroyd A.E. and Peres Y., Trees and matchings from point processes,
\emph{Electronic Communications in Probability} \textbf{8} (2003) 17--27.
\bibitem{Kallenberg} Kallenberg, O. \emph{Foundations of Modern Probability}, Springer.
%\bibitem{Kesten}  Kesten H. The critical probability of bond percolation on the square lattice equals 1/2.  Comm. Math. Phys. \textbf{74} (1980), no. 1, 41–-59.
%\bibitem{GaltonWatson}Watson H.W. and Galton F., "On the Probability of the Extinction of Families", \emph{Journal of the Anthropological Institute of Great Britain}, \textbf{4} (1875), 138–-144.
 \end{thebibliography}
\end{document}